\documentclass[12pt]{article}

\usepackage{amsmath}
\usepackage{amsthm}
\usepackage{amsfonts}
\usepackage{amssymb}
\usepackage{pb-diagram}

\newtheorem{Theorem}{Theorem}[section]
\newtheorem{Remark}[Theorem]{Remark}
\newtheorem{Corollary}[Theorem]{Corollary}
\newtheorem{Lemma}[Theorem]{Lemma}
\newtheorem{Proposition}[Theorem]{Proposition}
\newtheorem{Definition}[Theorem]{Definition}

\def\afooter#1#2{{\def\thefootnote{\mbox{${}^{#1}$}}\mbox{\footnotemark[0]}
        \footnotetext{#2} }}
        
\newcommand{\dlra}[1]{\stackrel{#1}{\longrightarrow}}
\newcommand\hra{\hookrightarrow}

\advance \topmargin by -\headheight
\advance \topmargin by -\headsep
     
\evensidemargin \oddsidemargin
\title{On the mapping class groups of $\#_r(S^p \times S^p)$ for $p = 3, 7$
\afooter{\mbox{ \ }}{Isotopy groups, self-homotopy equivalences of manifolds
\\ \hspace*{5mm} 2000 Mathematics
Subject Classification:  57R52, 55P10.}} 

\author{Diarmuid J. Crowley}
\date{\today}

\begin{document}
\maketitle

\begin{abstract}
For $M_r^{} := \sharp_r(S^p \times S^p)$, $p=3, 7$, we calculate $\pi_0{\rm Diff}(M_r)/\Theta_{2p+1}$ and $\mathcal{E}(M_r)$, respectively the group of isotopy classes of orientation preserving diffeomorphisms of $M_r$ modulo isotopy classes with representatives which are the identity outside a $2p$-disc and the group of homotopy classes of orientation preserving homotopy equivalences of $M_r$.
\end{abstract}

%\begin{abstract}
%We calculate the groups $\pi_0{\rm Diff}(\sharp_rS^p \times S^p)/\Theta_{2p+1}$ and $\mathcal{E}(\sharp_rS^p \times S^p)$ for $p = 3, 7$ and show in both cases that the fundamental extension splits if and only if $r = 1$.
%\end{abstract}
%%%%%%%%%%%%%%%%%%%%%%%%%%%%%%%%%%%%%%%%%%

\section{Introduction}
%%%%%%%%%%%%%%%%%%%%%%%%%%%%%%%%%%%%%%%%%%
Let $M$ be a closed smooth oriented $(k-1)$-connected $2k$-manifold with $k \geq 3$.  We consider the problem of understanding certain mapping class groups of $M$: for example $\pi_0{\rm Diff}(M)$, the isotopy classes of  orientation preserving diffeomorphisms of $M$ and $\mathcal{E}_{}(M)$, the homotopy classes of  orientation preserving homotopy equivalences of $M$.

We first simplify the problem for diffeomorphisms by considering the quotient
\[{\rm Aut}(M^{}) := \pi_0{\rm Diff}_{}(M^{})/\Theta_{2k+1}\]
where we regard $\Theta_{2k+1}$ as the group of isotopy classes of diffeomorphisms of the $2k$-disc  $D^{2k}$ fixed on the boundary so that extension by the identity from $D^{2k} \subset M$ to all of $M$ gives rise to an action of $\Theta_{2k+1}$ on $\pi_0{\rm Diff}(M)$. 

Assuming that $M$ is almost-parallelisable, results of Kreck \cite{Kre} lead to a short exact sequence with middle term $\rm{Aut}(M)$ and known kernel and cokernel.  Later Baues \cite{B} gave a similar sequence for $\mathcal{E}(M)$ and related the two sequences. This left extension problems which were solved in many cases (\cite{F, B, Kry2}) but which remain open in general.  

We focus on the case where $M$ is the $r$-fold connected sum of a product of spheres
\[ M_r^{2k} := \#_r(S^k \times S^k).\]
A fundamental invariant of a mapping class $[f: M_r^{2k} \to M_r^{2k}]$ is the induced automorphism $f_*$ of $H_k(M) : = H_k(M; \mathbb{Z})$.  Necessarily $f_*$ is also an automorphism of  the intersection form of $M_r^{2k}$
\[ \phi_r \colon H_k(M_r^{2k}) \times H_k(M_r^{2k}) \to \mathbb{Z}. \]
Specialising further to Hopf-invariant one dimensions $k = p = 3$ or $7$ we simply write $M_r$ for $M_r^{2p}$.  By \cite{Kre}[Theorem$\, 2$] and \cite{B}[Theorem$\, 10.3]$ the forgetful homomorphism $\mathcal{E} : {\rm Aut}(M_r) \rightarrow \mathcal{E}(M_r)$ fits into the following commuting diagram of extensions
\begin{equation} \label{exteqn}
\begin{diagram}
\divide\dgARROWLENGTH by 2
\node{0} \arrow{e} \node{H_p(M_r^{})} \arrow{e,t}{} \arrow{s,r}{} \node{{\rm Aut}_{}(M_r)} \arrow{e,t}{\alpha} \arrow{s,r}{\mathcal{E}} \node{{\rm Aut}(\phi_r)} \arrow{s,r}{=} \arrow{e} \node{1}\\
\node{0} \arrow{e} \node{H_p(M_r) \otimes \pi_{2p}(S^p)} \arrow{e,t}{} \node{\mathcal{E}_{}(M_r)} \arrow{e,t}{\alpha} \node{{\rm Aut}(\phi_r)} \arrow{e} \node{1}
\end{diagram}
\end{equation}
where $\alpha([f]) = f_*$, $\pi_{2p}(S^p) \cong \mathbb{Z}/12$ or $\mathbb{Z}/120$ as $p = 3$ or $7$ and the left vertical arrow is reduction mod $12$ or $120$.  Our main result is the following
\begin{Theorem} \label{mainthm}
The extensions in $(\ref{exteqn})$ split if and only if $r=1$.
\end{Theorem}

\begin{Remark}
{\em A complete algebraic description of the extensions in $(\ref{exteqn})$ appears in Corollary \ref{maincor}.}
\end{Remark}

\begin{Remark} \label{nonsplitrem}
{\em Despite Theorem \ref{mainthm}, \cite{F} and \cite{Kry2} show that $\pi_0{\rm Diff}(S^3 \times S^3) \rightarrow {\rm Aut}(\phi_1)$ does not split.  A complete algebraic description of $\pi_0{\rm Diff}(S^3 \times S^3)$ appears in \cite{Kry2}.}
\end{Remark}

\begin{Corollary} \label{actcor}
If $r >1$ the symplectic group ${\rm Aut}(\phi_r)$ cannot act, even up to homotopy, on $M_r$ such that the induced action on $H_p(M_r)$ is by ${\rm Aut}(\phi_r)$.  If $r=1$ and $p=3$ the same statement holds up to isotopy.
\end{Corollary}

\begin{proof}
For $r > 1$ such an action would provide a splitting for $\mathcal{E}(M_r) \rightarrow {\rm Aut}(\phi_r)$ which by Theorem \ref{mainthm}, does not exist.  For $r=1$ and $p=3$ such an action would provide a splitting for $\pi_0{\rm Diff}(S^3 \times S^3) \rightarrow {\rm Aut}(\phi_1)$ which by Remark \ref{nonsplitrem} does not exist.
\end{proof}

\begin{Remark} \label{splitrem}
{\em For $M_r^{2k}$ with $k = 4j-1$ and $k \neq 3, 7$, there are similar extensions for ${\rm Aut}(M)$ and $\mathcal{E}(M)$ except that ${\rm Aut}(\phi_r)$ must be replaced by ${\rm Aut}(\psi_r)$ where $\psi_r$ denotes a quadratic refinement of $\phi_r$ with trivial Arf invariant (see Section \ref{algsec}).  By \cite{Kry1} and \cite{B} the extensions for ${\rm Aut}(M_r^{2k})$ and $\mathcal{E}(M_r^{2k})$ always split.  Thus Theorem \ref{mainthm} is another peculiar feature of the existence of maps of Hopf-invariant one. }
%A point we discuss more in Section \ref{KBsubsec}.
\end{Remark}

The rest of the paper is organised as follows: we complete the introduction by recalling the results of \cite{Kre} and \cite{B} which establish the extensions in $(\ref{exteqn})$.  We also explain why the methods of \cite{B} and \cite{Kry1} which give the splittings mentioned in Remark \ref{splitrem} fail in dimensions $6$ and $14$.  In the algebraic Section \ref{algsec} we assemble facts about groups extensions, $1$-cocycles and quadratic forms: \!Corollary \ref{extcor} may be of independent interest.  Section \ref{topsec} contains the topology.  In Proposition \ref{s_Fprop1} we define a $1$-cocycle on ${\rm Aut}(M_r)$ called the global derivative and in Corollary \ref{maincor} the topological extensions of $(\ref{exteqn})$ are identified with the algebraic extensions of $(\ref{gammaexteqn})$ from Section \ref{algsec}, proving Theorem \ref{mainthm}. 
%%%%%%%%%%%%%%%%%%%%%%%%%%%%%%%%%%%%%%%%%%

\subsection{Theorems of Kreck and Baues} \label{KBsubsec}
%%%%%%%%%%%%%%%%%%%%%%%%%%%%%%%%%%%%%%%%%%
Let $M$ be a closed smooth $(k-1)$-connected almost-parallelisable smooth manifold with $k \geq 3$.   We let ${\rm Aut}(\alpha_M)$ denote the group of automorphisms of $H_k(M)$ which preserve the extended quadratic form $(H_k(M), \phi_M, \alpha_M)$ defined in \cite{Wa}[p.$\,165$]: $\phi_M$ is the usual intersection form and $\alpha_M : H_k(M) \rightarrow \pi_{k-1}(SO_k)$ gives the isomorphism class of the normal bundle of an embedded sphere representing a given element of $H_k(M)$.  Wall showed that every element of ${\rm Aut}(\alpha_M)$ is realised by a diffeomorphism of $M^\bullet : = M-{\rm int}(D^{2k})$.  

Building on \cite{Wa} and applying Cerf's results on psuedo-isotopy from \cite{C}, Kreck gave a pair of exact sequences in \cite{Kre}[Theorem$\, 2$]  which compute $\pi_0{\rm Diff}(M)$.  We adapt these sequences to the simpler case of ${\rm Aut}(M)$:
\begin{equation} \label{Kseqeqn1}
0 \dlra{} {\rm SAut}(M) \dlra{} {\rm Aut}(M) \dlra{\alpha} {\rm Aut}(\alpha_M) \dlra{} 1,
\end{equation}
\vskip -0.4cm
\begin{equation} \label{Kseqeqn2}
0 \dlra{} {\rm SAut}(M) \dlra{\chi} {\rm Hom}(H_k(M), S\pi_k(SO_k)) \dlra{} 0.
\end{equation}
\vskip 0.2cm
\noindent
Here ${\rm SAut}(M)$ is by definition the kernel of $\alpha$, here and elsewhere $S$ denotes stabilisation: in this case $S \colon \pi_k(SO_k) \to \pi_k(SO_{k+1})$ and we explain the isomorphism $\chi$ below.

The elements of ${\rm SAut}(M)$ have representatives $f : M \cong M$ which pointwise fix $k$-spheres representing generators of $H_k(M)$ and the map $\chi$ of $(\ref{Kseqeqn2})$ is defined by taking the derivative of $f$ restricted to the normal bundle of these $k$-spheres: we shall return to this point in the proof of Lemma \ref{s_Fprop1}.

The upper exact sequence of $(\ref{exteqn})$ is obtained from $(\ref{Kseqeqn1})$ and $(\ref{Kseqeqn2})$ as follows.  Firstly the extended quadratic form of $M_r$ is identical to the intersection form $\phi_r$ since $\pi_{p-1}(SO_p) = 0$.  Secondly, identifying
\[ S^2(\pi_p(SO_p)) \cong 2 \mathbb{Z} \subset \mathbb{Z} = \pi_p(SO)\]
we use $\chi$ and Poincar\'{e} duality in the following sequence of isomorphisms
\[ {\rm SAut}(M_r) \cong {\rm Hom}(H_p(M_r), 2 \mathbb{Z}) \cong H^p(M_r; 2 \mathbb{Z}) \cong H_p(M_r; 2 \mathbb{Z)} \cong H_p(M_r).\]

In general we can directly construct an isomorphism $H_k(M) \otimes S\pi_k(SO_k) \cong {\rm SAut}(M)$ using high dimensional Dehn-twists as follows.  Given $v \in H_k(M)$ and $\beta \in S\pi_k(SO_k)$, choose an embedding $V : S^k \times D^k \to M$ representing $v$ and an appropriate smooth map $B: (D^k, S^{k-1}) \rightarrow (SO_{k+1}, {\rm Id})$ representing $\beta$ so that the following map is a diffeomorphism
\begin{equation} \label{def:f-v-beta}
f(V, B) : M \cong M, ~~~ m \longmapsto \left\{\begin{array}{cccc} V(B(y)x, y), & m = V(x, y) \in {\rm Im}(V), \\ m & m \notin {\rm Im}(V). \\ \end{array}   \right.
\end{equation}
Applying $(\ref{Kseqeqn2})$ the following homomorphism is readily seen to be an isomorphism
\[ H_k(M; S\pi_k(SO_k)) \cong H_k(M) \otimes S\pi_k(SO_k) \longrightarrow {\rm SAut}(M), ~~~ v \otimes \beta \longmapsto [f(V, B)]. \]

We now discuss how maps of Hopf invariant one affect splitting results for ${\rm Aut}(M)$: the key point is that if $c_j$ is the order of the cokernel of the stabilisation map 
\[ S^2 : \pi_{4j-1}(SO_{4j-1}) \rightarrow \pi_{4j-1}(SO)\]
then $c_1 = c_2 = 2$ and $c_j = 1$ for $j > 2$.  Next we recall and correct a construction from \cite{Kre} bearing $c_j$ in mind.  Let $k=4j-1$.   Given a diffeomorphism $f : M \cong M$, let 
\[M_f := M \times [0, 1]/(0, m) \sim (1, f(m))\]
be the mapping torus of $f$, let $p_{j}(M_f) \in H^{4j}(M_f)$ be the $j$-th Pontrjagin class of $M_f$ and let $C : M_f \rightarrow \Sigma(M_+)$ be the map collapsing $M \times \{0 \} \subset M_f$ to a point where $\Sigma(M_+)$ denotes the suspension of $M$ plus a disjoint point.  As $\alpha([f]) = {\rm Id}$, the map $C$ induces an isomorphism $C^* : H^{4j}(\Sigma(M_+)) \cong H^{4j}(M_f)$.  Identifying 
\[ H^{4j}(\Sigma(M_+)) = H^{4j-1}(M) =  {\rm Hom}(H_{4j-1}(M), \pi_{4j-1}(SO))\] 
Kreck defines 
\[
p([f]) := C^{*-1}(p_{4j}(M_f)) \in  {\rm Hom}(H_{4j-1}(M), \pi_{4j-1}(SO)).\]

\begin{Lemma}[{Cf.\,\cite{Kre}[Lemma$\, 2$]}]
Let $a_j = (3 - (-1)^j)/2$.  The map $[f] \mapsto p([f])$ defines a homomorphism ${\rm SAut}(M) \rightarrow {\rm Hom}(H_{4j-1}(M), \pi_{4j-1}(SO))$ such that
\[p([f]) = \pm a_j c_j (2j-1)! \,\chi([f]).\]
%~ \in ~  {\rm Hom}(H_k(M), \pi_{4j-1}(SO))
\end{Lemma}

\begin{proof}
This is Lemma 2 of \cite{Kre} restated for ${\rm SAut}(M)$ as opposed to $\pi_0{\rm SDiff}(M)$, the analogous subgroup of $\pi_0{\rm Diff}(M)$, and corrected to include the factor $c_j$ which was omitted in \cite{Kre}.   There $S\pi_{4j-1}(SO_{4j-1})$ is identified with $\mathbb{Z}$ and then with $\pi_{4j-1}(SO)$: a move which is valid only if $c_j = 1$.  Kreck's proof of his Lemma 2 gives a correct proof of this lemma.
\end{proof}
%
%\noindent
%From $(\ref{Kseqeqn1})$ and $(\ref{Kseqeqn2})$ we obtain the extension
%
%\begin{equation} \label{autexteqn2}
%0 \longrightarrow  {\rm Hom}(H_{4j-1}(M), \pi_{4j-1}(SO)) \dlra{I} {\rm Aut}(M) \dlra{H_k} {\rm Aut}(\alpha_M) \longrightarrow 1.
%\end{equation}
%
In \cite{Kry1}[Ch.$\, 3$] Krylov extended the definition of $p$ to all of ${\rm Aut}(M)$ and in Theorem 3.2 of \cite{Kry1} he used the extended function $p$ and \cite{Kre}[Lemma 2] to construct a splitting of $(\ref{Kseqeqn1})$.  This elegant argument works perfectly when $c_{j} = 1$ and so for $k \neq 3, 7$ $(\ref{Kseqeqn1})$ splits.  However, when $k = 3, 7$ the existence of $p$ is not enough to split $(\ref{Kseqeqn1})$ in general as we prove in Section \ref{topsec}.

We next turn to the homotopy category and results of \cite{B}.  Since $M$ is $(k-1)$-connected there is a homotopy equivalence
\[ M \simeq M^\bullet \cup_h D^{2k}\]
where $M^\bullet = M - {\rm int}(D^{2k})$ retracts to a $k$-skeleton of $M$, $M^\bullet \simeq \vee S^k$ and $h: S^{2k-1} \rightarrow M^\bullet$ is the attaching map for the top cell of $M$.  

Recall that $\mathcal{E}(M)$\footnote{Note that as we restrict our attention to orientation preserving maps we write $\mathcal{E}(M)$ for what Baues calls $\mathcal{E}_{+}(M)$.} denotes the group of homotopy classes of orientation preserving homotopy equivalences of $M$. The fundamental extension for $\mathcal{E}(M)$ is the short exact sequence
\begin{equation} \label{fundexteqn}
0 \dlra{} \mathcal{E}(M \, | \, M^\bullet) \dlra{} \mathcal{E}(M) \dlra{} \mathcal{E}(M^\bullet, h) \dlra{} 0
\end{equation}
where $\mathcal{E}(M \, | \, M^\bullet)$ is the subgroup of homotopy classes containing representatives which are the identity on $M^\bullet$, $\mathcal{E}(M^\bullet, h)$ is the group of homotopy classes of homotopy equivalences of $M^\bullet$ compatible with the attaching map $h$ and $\mathcal{E}(M \, | \, M^\bullet)$ is an $\mathcal{E}(M^\bullet, h)$-module whose underlying group is finite abelian by \cite{B}[Theorem$\, 1.4$].  The lower exact sequence of $( \ref{exteqn} )$ is the fundamental extension for $M_r$ since by \cite{B}[Theorem$\, 7.6$]
\[ \mathcal{E}(M_r^\bullet, h) \cong {\rm Aut}(\phi_r) \text{~~and~~} \mathcal{E}(M_r \, | \, M_r^\bullet) \cong H \otimes S\pi_{2p}(S^p)\]
and moreover $S : \pi_{2p}(S^p) \to \pi_{2p+1}(S^{p+1})$ is an injection. 

The elements of $\mathcal{E}(M \, | \, M^\bullet)$ can be realised by so-called ``pinch maps": given an element $v \in H_k(M) \cong \pi_k(M) $ and $\gamma \in \pi_{2k}(S^k)$ we define the self-homotopy equivalence 
\[ f(v, \gamma) : =  ({\rm Id}_M \vee v (\circ \gamma)) \circ p : M \simeq M \]
where $p : M \to M \vee S^{2k}$ is the map which contracts $\partial D^{2k}$ to a point and we use $\gamma$ and $v$ to denote representatives of the homotopy classes they define.  It is not hard to see that 
\[ \mathcal{E}([f(v, S\beta)]) = [f(v, JS\beta)] \in \mathcal{E}(M) \] 
where $J : S\pi_{k}(SO_{k}) \rightarrow S\pi_{2k}(S^{k})$ is the $J$-homomorphism.  This observation leads directly to that part of \cite{B}[Theorem$\, 10.3$] which identifies the composition 
\[\mathcal{E}|_{{\rm SAut}(M_r)} \circ \chi^{-1}: {\rm Hom}(H_p(M_r), S\pi_p(SO_p)) \to H_p(M_r) \otimes S\pi_{2p}(S^p)  \]
as the composition of Poincar\'{e} duality ${\rm PD} : {\rm Hom}(H, S\pi_k(SO_k)) \rightarrow H \otimes S\pi_k(SO_k)$ with the $J$-homomorphism applied to the second factor: ${\rm Id} \otimes J: H \otimes S\pi_k(SO_k)) \rightarrow H \otimes S\pi_{2k}(S^{k})$.  In particular for $M = M_r$, the $J$-homomorphism is isomorphic to reduction mod $c$ and so $\mathcal{E} : {\rm Aut}(M_r) \to \mathcal{E}(M_r)$ is surjective.

We now discuss how the existence of maps of Hopf invariant one effects Baues' results on splitting in the fundamental extension.  Call $M$ $\Sigma_1$-reducible if the suspension $\Sigma M$ is homotopic to $S^{2k+1} \vee \Sigma M^\bullet$.  For example $M_r^{2k} = \sharp_r(S^k \times S^k)$ is $\Sigma_1$-reducible with 
\[ \mathcal{E}(M_r^{2k} \, | \, (M_r^{2k})^{\bullet}) \cong S\pi_{2k}(\vee_{2r} S^k).\]  
Theorem 5.2 of \cite{B} gives a criterion for when the fundamental extension of $M$ splits: one requires that the inclusion $S\pi_{2k}(\vee_{2r} S^k) \subset \pi_{2k+1}(\vee_{2r} S^{k+1})$ admits a retraction of $\mathcal{E}((M_r^{2k})^\bullet, h)$-modules.  If $k \geq 3$ is odd and $k \neq 3, 7$ this criterion holds but for $M_r$ it fails: the inclusion $\pi_{2p}(\vee_{2r} S^p) \rightarrow \pi_{2p+1}(\vee_{2r} S^{p+1})$ is isomorphic to $\oplus_{2r} \mathbb{Z}_d \hra \oplus_{2r}(\mathbb{Z} \oplus \mathbb{Z}_d)$ where $d = 12, 120$ if $p=3, 7$ and there is a retraction as abelian groups but not of $\mathcal{E}(M_r^\bullet, h)$-modules.  The point being that $-{\rm Id}$ acts by $-1$ on $\mathbb{Z}_d  \subset \pi_{2p+1}(S^p)$ but sends $(1, 0) \in \pi_{2p+1}(S^p)$ to $(1, -1)$ where $(1, 0)$ is the Hopf map and $(0, 1)$ is a suitable generator of $S\pi_{2p}(S^p)$: see \cite{B}[\S 9].  The case of $M_1 = S^p \times S^p$ shows that the splitting criterion of \cite{B}[Theorem$\, 5.2$] is not a necessary condition for splitting.

We conclude the introduction by correcting the statements in \cite{B} pertaining to splitting the fundamental extension for $M_r$.  Note that since every $M_r$ has an orientation reversing map of order two we may confine our attention to $\mathcal{E}(M_r)$.  Theorem 6.3 of \cite{B} computes $\mathcal{E}(S^k \times S^k)$ and is correct as stated but the proof of splitting given fails for $k = 3, 7$ in which case a correct proof is given in Corollary \ref{maincor} below and can also be found in \cite{Kry2} for $k = 3$.   The final sentence of  \cite{B}[Theorem$\, 8.14$] should be altered to exclude the dimensions $k = 3, 7$ and the correct statement is found in Corollary \ref{maincor} since for $k=3$ or $7$ the only $\Sigma_1$-reducible $(k-1)$-connected Poincar\'{e} complexes of dimension $2k$ are homotopy equivalent to $M_r$ for some $r$.

\noindent
{\bf Acknowledgements:} I would like to thank Nikolai Krylov for helpful discussions.  I would also like to thank the referee for comments which clarified a number of mathematical points and which significantly improved the organisation and presentation of our ideas.
%%%%%%%%%%%%%%%%%%%%%%%%%%%%%%%%%%%%%%%%%%

\section{$1$-cocycles, quadratic forms and group extensions} \label{algsec}
%%%%%%%%%%%%%%%%%%%%%%%%%%%%%%%%%%%%%%%%%%
Let $S$ be a finitely generated group and let $K$ be an abelian group equipped with an $S$-action, inducing an $S$-module structure on $K$. In this section we consider certain group extensions of $S$ by $K$:
\[0 \to K \to \Gamma \to S \to 1\]
related to a given short exact sequence $0 \to K \to L \to M \to 0$ of $S$-modules. The group extensions of interest are subextensions of the trivial extension $L \rtimes S \to S$ and so are defined by $1$-cocyles on $S$ with values in $M$ as we now explain.  

Fix an $S$-module $X$; i.e. an action of $S$ on an abelian group $X$.  A $1$-cocycle is a map $s:S \rightarrow X$ satisfying
\begin{equation} \label{eqn:1-cocy-prop}
s(AB) = s(A) \cdot B + s(B) ~~\forall \, A, B \in S.
\end{equation}
Note that necessarily $s({\rm Id}) = 0 \in X$ where ${\rm Id} \in S$ is the identity.
\begin{Definition} 
The set of $1$-cocycles, ${\rm Z^1}(S; X)$, is an abelian group under
pointwise addition.  Every $x \in X$ defines a $1$-cocycle by
$$s(x)(A) := x \cdot A - x$$
and the set of such $1$-cocycles forms a subgroup ${\rm B}^1(S; X) \subset {\rm Z^1}(S; X)$.  The first cohomology group of $S$ with coefficients in $X$ is by definition
\[H^1(S; X) : = {\rm Z^1}(S; X)/{\rm B^1}(S; X).\]
\end{Definition}

Using the action of $S$ on $X$ we obtain the usual semi-direct product $X \rtimes S$ with multiplication law
\[(x, A) \cdot (y, B) = (x + y \cdot A, AB).\]
Returning to the short exact sequence $K \to L \to M$ of $S$-modules we have
\begin{Definition}[{Cf.\,\cite{Br}[Ex.\,2, p.\,95]}] \label{gammasdef}
Given a $1$-cocycle $s \in Z^1(S; M)$ and writing $\bar x$ for the image of $x \in L$ in $M$, we define $\Gamma(s)$, a subgroup of $L \rtimes S$, by 
\[ \Gamma(s) : = \{ (x, A) \, | \, \bar x = s(A) \in M \} \subset L \rtimes S.\]
\end{Definition}

Note that $(\ref{eqn:1-cocy-prop})$ ensures that $\Gamma(s)$ is indeed a subgroup.  Note also that by construction the projection map $\Gamma(s) \to S$, $(x, A) \mapsto A,$ fits into the following commutative diagram of extensions of $S$.
\[ \begin{diagram}
\divide\dgARROWLENGTH by 2
\node{0} \arrow{e} \node{K} \arrow{e} \arrow{s,J} \node{\Gamma(s)} \arrow{e} \arrow{s} \node{S} \arrow{s,r}{=} \arrow{e} \node{1}\\
\node{0} \arrow{e} \node{L} \arrow{e} \node{L \rtimes S} \arrow{e} \node{S} \arrow{e} \node{1}
\end{diagram} \]

Recall that two extensions $\Gamma_1 \to S$ and $\Gamma_2 \to S$ are isomorphic is there is an isomorphism $\Gamma_1 \cong \Gamma_2$ commuting with the homomorphisms to $S$.  The following proposition summarises basic results from the theory of abelian group extensions which we need to analyse the extension $\Gamma(s) \to S$.  
\begin{Proposition} [{Cf.\,\cite{Br}[\S 3 Ch.$\,\text{IV}$]}] \label{prop:Brown+}
Let Ê$0 \to K \to L \to M \to 0$ be a short exact sequence of $S$-modules.
\begin{enumerate}
\item[\em{(i)}] Group extensions $K \to \Gamma \to S$ are classified up to isomorphism by $H^2(S; K)$.
\item[\em{(ii)}] Associated to the short exact sequence $K \to L \to M$ of $S$-modules there is a long exact sequence in cohomology
\[ \dots \dlra{} H^1(S; L) \dlra{} H^1(S; M) \dlra{\delta} H^2(S; K) \dlra{} H^2(S; L) \to \dots ~ .\]
\item[\em{(iii)}]The isomorphism class of the the extension $\Gamma(s) \to S$ in Definition \ref{gammasdef} is determined by $[\Gamma(s)] = \delta([s]) \in H^2(S; K)$.  In particular two extensions $\Gamma(s) \to S$ and $\Gamma(s') \to S$ are isomorphic if $[s] = [s'] \in H^1(S; M)$.
\item[\em{(iv)}] Given an extension $K \to \Gamma \to S$ let $\Gamma$ act on $K$ via the projection $\pi: \Gamma \to S$.  This extension is isomorphic to $K \to \Gamma(s) \to S$ for some $s \in H^1(Z; X_2)$ if and only if there is a $1$-cocycle $d : \Gamma \to M$ such that $d|_K: K \cong K$ is an isomorphism of $S$-modules.  In this case $\Gamma \cong \Gamma(\bar d)$ where $\bar d: S \to M$ is induced by $d$.
\end{enumerate}
\end{Proposition}

\begin{proof}
Part $\text{(i)}$ is \cite{Br}[Theorem 3.2, Ch.$\,\text{IV}$], $\text{(ii)}$ is well-known and $\text{(iii)}$ is found in \cite{Br}[Ex.$\, 2$, Ch.$\,\text{IV}$, \S 3, p.\,95].  For the `if' of $\text{(iv)}$ observe that the map $d \times \pi : \to L \rtimes S$ defines an isomorphism onto $\Gamma(\bar d)$.  For the `only if' part of $\text{(iv)}$, define $d : \Gamma(s) \to L$ to be the projection $(x, A) \mapsto x$.
\end{proof}

We next consider examples relevant to ${\rm Aut}(M_r)$.  Let $H$ be a free abelian group of rank $2r$, let $\phi_r : H \times H \rightarrow \mathbb{Z}$ be a nonsingular, anti-symmetric bilinear form and let $S = {\rm Aut}(\phi_r)$ be the group of automorphisms of $\phi_r$: $S$ is thus the symplectic group $S \cong {\rm Sp}(2r, \mathbb{Z})$.  Let $H^* = {\rm Hom}(H, \mathbb{Z})$ be the dual of $H$ where ${\rm Aut}(\phi)$ acts on the right of $H^*$ by composition:
\[ H^* \times {\rm Aut}(\phi_r) \longrightarrow H^*, ~~~ (x, A) \longmapsto x \cdot A := x \circ A.\]

Now let $c$ be a positive integer or infinity and denote by $H^*_c := H^*/cH^*$, $H^*_\infty := H^*$.   There are obvious induced actions of $S$ on $H^*_c$ and a short exact sequence of $S$-modules
\begin{equation} \label{eqn:Hstarext}
 0 \dlra{} H^*_c \dlra{\iota} H^*_{2c} \dlra{\rho} H^*_2 \dlra{} 0
 \end{equation}
where $\iota$ is multiplication by $2$ interpreted on generators.
\begin{Definition}
 The semi-direct product $H_c^* \rtimes {\rm Aut}(\phi_r)$ defined by the above action is called the {\em Jacobi group} and denoted by $\Gamma(\phi_r, c)$.
\end{Definition}
Recall that a quadratic refinement of the hyperbolic form $\phi_r: H \times H
\rightarrow \mathbb{Z}$ is a function $\psi: H \rightarrow \mathbb{Z}_2$ such that
\[\psi(v + w) = \psi(v) + \psi(w) + \bar \phi_r(v,w) \]
where $\bar \phi_r$ is the mod $2$ reduction of $\phi$.  Let $\Psi$ be the set of all such $\psi$ refining $\phi$.  The symplectic group ${\rm Aut}(\phi_r)$ acts on $\Psi$ by $\psi \cdot A = \psi \circ A^{}$ and we shall use this action to construct a $1$-cocyle on ${\rm Aut}(\phi_r)$ with values in $H^*_2$.
\begin{Lemma} \label{keylem}
\begin{enumerate}
\item[\em{(i)}]
For all $\psi_0, \, \psi_1 \in \Psi$ the function $\psi_1 - \psi_0$ belongs to $H^*_2$ and moreover $H^*_2$ acts freely and transitively on $\Psi$ by addition of functions.
\item[\em{(ii)}]The orbits of the action of ${\rm Aut}(\phi_r)$ on $\Psi$ are the isomorphism classes of
quadratic refinements of $\phi$.  For each $r$ there are two orbits which are detected by
the Arf invariant, $\Psi = \Psi_0 \sqcup \Psi_1, |\Psi_0| = 2^{2r-1}+2^{r-1}, 
|\Psi_1| = 2^{2r-1} - 2^{r-1}$.
\item[\em{(iii)}]Fixing $\psi \in \Psi$ defines a $1$-cocycle $s(\psi): {\rm Aut}(\phi_r) \rightarrow H^*_2$ where
\[ s(\psi)(A) := \psi \cdot A - \psi. \]
\item[\em{(iv)}]For all $\bar{x} \in H^*_2, ~s(\psi + \bar{x}) = s(\psi) + s(\bar{x}).$
\end{enumerate}
\end{Lemma}

\begin{proof}
The first two statements are well known so we omit their proof as well as the definition of the
Arf invariant.  The third and fourth statements statement are easily checked.  For example:
\[ \begin{array}{cl}
s(\psi)(AB) & = \psi \cdot (AB) - \psi \\
& = (\psi \cdot A - \psi) \cdot B + \psi \cdot B - \psi\\
& = s(\psi)(A) \cdot B + s(\psi)(B).
\end{array} \]
%\[ \begin{array}{cl}
%s_{\psi + \bar x}(A) & = (\psi + \bar x) \circ A - \psi -  \bar x \\ 
%& = s_\psi(A) + s_{\bar x}(A).
%\end{array} \]
\end{proof}

Given a quadratic form $\psi$ we write $\Gamma(\psi, c)$ for $\Gamma(s(\psi))$ where $s(\psi)$ is the $1$-cocycle of Lemma \ref{keylem} and $\Gamma(s(\psi))$ is defined as in Definition \ref{gammasdef} using the short exact sequence $(\ref{eqn:Hstarext})$.  Thus there is an extension
\begin{equation} \label{gammaexteqn}
0 \dlra{} H^*_c \dlra{} \Gamma(\psi, c) \dlra{\pi} {\rm Aut}(\phi_r) \dlra{} 1
\end{equation}
determined by $\delta_c[s(\psi)] \in H^2(S; H^*_c)$ where $\delta_c$ is the boundary map $H^1(S; H_c^*) \to H^2(S; H^*_2)$.  
%{\em Henceforth we assume that if $c$ is finite then $c$ is even.} 
The $1$-cocycles $s(\psi)$ have the following key properties

\begin{Lemma} \label{lem:s(psi)}
If $c = 2d$ or $\infty$ then for all quadratic forms $\psi : H \to \mathbb{Z}_2$ refining $\phi_r$,
\begin{enumerate}
\item[\em{(i)}]
$[s(\psi)] = 0 \in H^1(S; H_2^*)$ if and only if $r = 1$,  
\item[\em{(ii)}]
$\delta_c([s(\psi)])= 0 \in H^2({\rm Aut}(\phi_r); H^*_c)$ if and only if $[s(\psi)] = 0$.
\end{enumerate}
\end{Lemma}
\noindent
As immediate consequences we have the following two statements.

\begin{Proposition} \label{keyextprop}
If $c = 2d$ or $\infty$ then the extension $\Gamma(\psi, c) \dlra{\pi}  {\rm Aut}(\phi_r)$ in $(\ref{gammaexteqn})$ splits if and only if $r = 1$.
\end{Proposition}

\begin{proof}
By Proposition \ref{prop:Brown+} $\text{(iii)}$ this extension splits if and only if $\delta_c([s(\psi)]) = 0$ which by Lemma \ref{lem:s(psi)} happens if and only if $r =1$.
\end{proof}

\begin{Corollary} \label{extcor}
Identify the groups ${\rm Aut}(\phi_r) = Sp(2r, \mathbb{Z})$ and $H^*_c = (\mathbb{Z}/c)^{2r}$ with the standard $Sp(2r, \mathbb{Z})$ action.  If $c = 2d$ or $\infty$ then for all $r \geq 2$, the class $\delta_c([(s(\psi)])$ generates a $\mathbb{Z}/2$-summand of $H^2(Sp(2r,\mathbb{Z}); (\mathbb{Z}/c)^{2r})$.
\end{Corollary}

\begin{proof}[Proof of Lemma \ref{lem:s(psi)}]
To prove $\text{(i)}$  note that by Lemma \ref{keylem} $\text{(iv)}$, $[s(\psi)] = 0$ if and only if there exists $\bar x \in H^*_2$ such that $s(\psi + \bar x)$ is identically zero.  But for any $\chi \in \Psi$, $s(\chi)$ is identically zero if and only if the orbit $\chi \cdot {\rm Aut}(\phi)$ contains precisely one element.  By Lemma \ref{keylem} $\text{(ii)}$ this happens if and only if $\chi$ is the Arf invariant $1$ quadratic refinement and $r=1$.  

To prove $\text{(ii)}$ it is enough by Proposition \ref{prop:Brown+} $\text{(ii)}$ to prove that $[s(\psi)] = 0$ if and only if $[s(\psi)]$ lies in the image of  $\rho : H^1({\rm Aut}(\phi), H^*_{2c}) \rightarrow H^1({\rm Aut}(\phi), H^*_2)$.  Hence we assume that $s_{2c} \in Z^1({\rm Aut}(\phi), H^*_{2c})$ is such that $\rho[s_{2c}] =  [s(\psi)]$.  First observe that for any quadratic form $\chi \in \Psi$ refining $\phi_r$ and any $v \in H$
\[0 = \chi(0) = \chi(v -v) = \chi(v) + \chi(-v) + \bar \phi_r(v, -v) = \chi(v) + \chi(-v) ~ \in ~ \mathbb{Z}_2.\]
Hence $-{\rm Id}$ is an automorphism of every $\chi$ and $s(\chi)(-{\rm Id}) = 0$ for all $\chi \in \Psi$.  Since $s_{2c}$ maps to some $s(\chi)$ under $\rho$, $s_{2c}(-{\rm Id}) = 2 x$ for some $x \in H^*_{2c}$.  But for any $1$-cocycle $s$ with values in $H^*_{2c}$ and for any $A \in S$,
\[ s(-A) = -s(A) + s(-{\rm Id})~~~\text{and}~~~s(-A) = s(-{\rm Id}) \cdot A + s(A) \]
and so $2s(A)  =  - (s(-{\rm Id}) \cdot A - s(-{\rm Id}))$.  Hence $2(s_{2c}(A) + s(x)(A)) = 0$.  It follows that $s_{2c}(A) + s(x)(A) \in \iota(H^*_c)$: recall that $c = 2d$ or $\infty$.  Thus $[s_{2c}] = \iota_*[s_{c}]$ for some $[s_{c}] \in H^1({\rm Aut}(\phi), H^*_c)$ and so $[s(\psi)] = \rho_*(\iota_*([s_{c}])) = 0$.
\end{proof}
%%%%%%%%%%%%%%%%%%%%%%%%%%%%%%%%%%%%%%%%%%

\section{The global derivative and ${\rm Aut}(M_r)$} \label{topsec}
%%%%%%%%%%%%%%%%%%%%%%%%%%%%%%%%%%%%%%%%%%
Recall that $M_r = \sharp_r(S^p \times S^p)$ for $p = 3, 7$.  For brevity we shall write $H$ for $H_p(M_r)$ %and we also identify $\pi_p(SO) =\mathbb{Z}$, $S^2\pi_p(SO_p) = 2\mathbb{Z}$ and hence ${\rm Hom}(H, \pi_p(SO)) = H^*$.  
so that the top extension of $(\ref{exteqn})$ is written
\[ 0 \longrightarrow H \dlra{} {\rm Aut}(M_r) \dlra{\alpha} {\rm Aut}(\phi_r) \longrightarrow 1. \]

For every manifold $M_r$ there there are stable framings of $\tau(M_r)$, the tangent bundle of $M_r$, which are bundle isomorphisms $F: \tau(M_r) \oplus \varepsilon(M_r) \cong \varepsilon^{2p+1}(M_r)$ where $\varepsilon^k(M_r) = M_r \times \mathbb{R}^k$ is the trivial rank $k$ vector bundle over $M_r$.  We begin this section by showing how a stable framing of the tangent bundle of $M_r$ defines a $1$-cocycle 
\[ s_F : {\rm Aut}(M_r) \rightarrow H^*\]  
where we recall that $H^* = {\rm Hom}(H; \mathbb{Z})$.  As we mentioned in Subsection \ref{KBsubsec}, this idea goes back to \cite{Kry2}[Ch.$\, 3$].  We take a different approach from \cite{Kry2} using what we call the global derivative.

For a bundle $E$ over $M_r$ we write $E_m$ for the fibre of $E$ over $m$.  Fix a stable framing $F$ of $M_r$ and for each $m \in M_r$ let $F_m$ be the isomorphism $(\tau(M_r) \oplus \varepsilon(M_r))_m \cong \mathbb{R}^{2p+1}$ defined by $F$.  Given a diffeomorphism $f:M_r \cong M_r$ let $D(f) : \tau(M_r) \cong \tau(M_r)_{}$ denote the derivative of $f$ with $D_m(f) : \tau(M_r)_m \cong \tau(M_r)_{f(m)}$.  Now for each $m \in M_r$ define $s_F(m,f) \in GL_{2p+1}(\mathbb{R})$, by 
\[s_F(m,f) := F_{f(m)}^{} \circ (D_m(f) \oplus {\rm Id}) \circ (F_{m})^{-1}: \mathbb{R}^{2p+1} \cong \mathbb{R}^{2p+1}.\]
Choosing a retraction from $GL_{2p+1}(\mathbb{R})$ to $SO(2p+1)$ and composing with the inclusion $SO(2p+1) \hra SO$, we obtain a continuous map $s_F(f): M \rightarrow SO$.  Evidently, the homotopy class of $s_F(f)$ is independent of the isotopy class of $f$ and so we obtain a map $s_F: \pi_0{\rm Diff}(M_r) \rightarrow [M_r,SO]$ called the global derivative.  Elementary obstruction theory shows that $[M_r, SO] \cong [M^\bullet_r, SO] \cong H^p(M_r;\pi_p(SO)) = H^*$.  If $f$ is the identity outside a $2p$-disc then $s_F([f])$ will be constant over a $p$-skeleton of $M$ and hence $s_F$ descends to a map $s_F:{\rm Aut}(M_r) \rightarrow H^*$.  Note that the set of homotopy classes $[M_r, SO]$ is a group under pointwise composition in $SO$ and that ${\rm Aut}(M_r)$ acts on the right of $[M_r, SO]$ by pre-composition of functions.  Moreover there is a canonical isomorphism $[M_r, SO] \cong H^*$ preserving ${\rm Aut}(M_r)$ actions so that for $x \in [M_r, SO] \cong H^*$ and $[g] \in {\rm Aut}(M_r)$, $x \circ [g] = x \cdot \alpha([g])$.

\begin{Proposition} \label{s_Fprop1}
The map $s_F:{\rm Aut}(M_r) \rightarrow H^*$ is a $1$-cocycle with $s_F|_{H} : H \cong 2H^*$.
\end{Proposition}

\begin{proof}
Let $f,g:M_r \cong M_r$ be diffeomorphisms.  Applying the chain rule at $m \in M_r$ we have
\[ \begin{array}{cl}
s_F(m,f \circ g) &= F_{f(g(m))}^{} \circ (D_m(f \circ g) \oplus {\rm Id}) \circ (F_{m})^{-1} \\
& = F_{f(g(m))} \circ (D_{g(m)}(f) \oplus {\rm Id}) \circ (D_{m}(g) \oplus {\rm Id}) \circ (F_{m})^{-1} \\
& = F_{f(g(m))} \circ (D_{g(m)}(f) \oplus {\rm Id}) \circ (F_{g(m)})^{-1} \circ F_{g(m)}^{} \circ (D_{m}(g) \oplus {\rm Id}) \circ (F_{m})^{-1} \\
& =  s_F(g(m),f) \circ s_F(m,g).
\end{array}\]
As the group structure on $[M_r, SO] \cong H^*$ is point-wise composition we have
\[s_F([f \circ g]) = s_F([f]) \circ [g] + s_F([g]) = s_F([f]) \cdot \alpha([g[)  + s_F([g]).  \]

For the second statement, note that $s_F$ is a homomorphism on the subgroup $H$ and so we may check this statement for a basis of $H$.  We do this now by choosing Dehn twists to represent a basis of $H$.  For $i  = 1, \dots, r$, choose embeddings $U_i, V_i : S^p \times D^p \to M_r$ representing a symplectic basis $\{ u_1, v_1, \dots, u_{r}, v_r \}$ for $H$ and choose an appropriate smooth map $(D^p, S^{p-1}) \rightarrow (SO_{p+1}, {\rm Id})$ representing a generator $\gamma$ of $S\pi_p(SO_p)$ so that $f(U_i, A)$ and $f(V_i, A)$ as defined in $(\ref{def:f-v-beta})$ in Section \ref{KBsubsec} are diffeomorphisms.  Writing $[f(u_i, \gamma)]$ for the mapping class of $f(U_i, A)$ and similarly for $[f(v_i, \gamma)]$ and also letting $\delta_{ij}$ denote the Kronecker delta, it is easy to check the following equalities: 
\[ s_F([f(u_j, \gamma)])(u_i) = 0 ~~~\text{and}~~~s_F([f(u_j, \gamma)])(v_i) = \delta_{ij} S\gamma \in \pi_p(SO).\] 
Hence $s_F([f(u_j, \gamma)])$ is the basis element of $2H^*$ defined by $\chi([f(u_j, \gamma)])$ where $\chi$ is the isomorphism of $(\ref{Kseqeqn2})$ again from Section \ref{KBsubsec}.  The same argument holds also for the basis vectors $v_j$ and we are done.
\end{proof}

Recall from Section \ref{algsec} that $\Gamma(\phi_r, \infty) = H^* \rtimes {\rm Aut}(\phi_r)$ is the Jacobi group of $\phi_r$.  As an immediate consequence of Proposition \ref{prop:Brown+} $\text{(iv)}$ and its proof we obtain

\begin{Corollary} \label{s_FxH_pcor}
Together $s_F$ and $\alpha$
% : {\rm Aut}(M_r) \rightarrow {\rm Aut}(\phi_r)$ 
define an injective homomorphism 
\[
\begin{array}{ccccc}
s_F \times \alpha : & {\rm Aut}(M_r) & \longrightarrow & \Gamma(\phi_r, \infty).\\
%& [f] & \mapsto & (s_F([f]),H_p([f])).
\end{array}\]
In particular ${\rm Aut}(M_r) \cong \Gamma(\psi_F, \infty)$ as defined in $(\ref{gammaexteqn})$ and is determined up to isomorphism by $[\bar s_F] \in H^1({\rm Aut}(\phi_r); H^*_2)$ where $\bar s_F$ is the $H^*_2$-valued $1$-cocyle on ${\rm Aut}(\phi_r)$ induced by $s_F$.
\end{Corollary}

It remains to determine $\bar s_F \in Z^1({\rm Aut}(\phi_r); H^*_2)$ and for this we use a small piece of the apparatus of surgery.  As well as $s_F$, a stable framing $F$ on $M_r$ defines a bundle map $(\bar c, c) : \tau(M_r) \oplus \varepsilon(M_r) \rightarrow \varepsilon^{2p+1}(S^{2p})$ to the trivial bundle over $S^{2p}$ which covers the collapse map $c: M_r \rightarrow S^{2p}$.  The map $(\bar c, c)$ is called a degree one normal map with respect to the tangent bundle in \cite{Lu}[Definition 3.50] and in \cite{Lu}[\S 4.4] it is proven that $(\bar c, c)$ equips $\pi_p(M_r) \cong H_p(M_r) = H$ with a quadratic form  $\psi(\bar c, c) : H \rightarrow \mathbb{Z}_2$ which we shall denote by $\psi_F$.  The quadratic form $\psi_F$ refines $\phi_r : H \times H \rightarrow \mathbb{Z}$ and is defined as follows: the choice of framing $F$ defines a unique regular homotopy class of immersions $V : S^p \rightarrow M_r$ for each $v \in H$ and $\psi_F(v) \in \mathbb{Z}_2$ is the number of self-intersections of a generic representative of this regular homotopy class counted modulo $2$. 

If we modify the stable framing $F$ we may modify the quadratic form $\psi_F$ as follows: any two stable framings $F_0$ and $F_1$ differ up to homotopy by a homotopy class $x = x(F_0, F_1) \in [M_r, SO] \cong H^*$.  If $\bar x \in H^*_2$ is the mod $2$ reduction of $x$ then we have

\begin{Lemma} \label{surgdifflem}
Let $F_0$ and $F_1$ be stable framings of $M_r$ as above.  Then $\psi_{F_0} = \psi_{F_1} + \bar x$.
\end{Lemma}

\begin{proof}
This is a standard, if subtle, result of surgery: see \cite{Le}[Theorem$\, 4.2 (b)$] and the discussion preceding it for a proof.  Of course Levine uses normal maps over the stable normal bundle but the bijective correspondence between these and tangential normal maps is covered in detail in \cite{Lu}[Lemma 3.51].
\end{proof}
\begin{Lemma} \label{s_Flem2} 
Let $F$ be a stable framing of $M_r$ with associated quadratic form $\psi_F$.  Then $s_F = s({\psi_F}) \circ \alpha$ {\rm mod} $2$.
\end{Lemma}

\begin{proof}
%The second identity follows immediately from the first and the relevant definitions so we prove the first.  Let $F$ be a stable framing of $M_r$.  As $s_F \circ i$ vanishes mod $2$ we see that $s_F([f])$ mod $2$ depends only on $H_p([f])$.  The second statement follows immediately from the first, Lemma \ref{s_FxH_pcor} and the definition of $\Gamma(\psi_F)$ in $(\ref{gammadefeqn})$ so we prove the first.  
Given $[f] \in {\rm Aut}(M_r)$ recall that $\alpha([f]) = f_* : H \cong H$.  We must show for all $[f]$ and for all $v \in H$ that $\psi_F(f_*(v)) - \psi_F(v) = s_F([f])(v)$ mod $2$.  Let $V: S^p \rightarrow M_r$ be an immersion which represents $v$ and which lies in the regular homotopy class prescribed by $F$.  As $f$ is a diffeomorphism, $V$ and $f \circ V$ have the same number of double points.  By Lemma \ref{surgdifflem} and the discussion preceding it, we see that $\psi_F(f_*(v)) - \psi_F(v)$ equals the difference mod $2$ of the framings prescribed by $F \circ D(f)$ and $F$ over the homotopy class $v$.  That is,  $\psi_F(f_*(v)) - \psi_F(v)  = s_F([f])(v)$ mod $2$.
%By the definition of $s_F$, it follows that  $f \circ V$ is in the regular homotopy class prescribed by $F$ if and only if $s_F([f])(v) = 0$ mod $2$.  
\end{proof}
%%%%%%%%%%%%%%%%%%%%%%%%%%%%%%%%%%%%%%%%%%

\begin{Corollary} \label{maincor}
A stable framing $F$ defines isomorphisms from the extensions of $(\ref{exteqn})$,
%\[ 0 \dlra{} {\rm Hom}(H, S\pi_p(SO_p)) \dlra{i} {\rm Aut}(M_r) \dlra{H_p} {\rm Aut}(\phi_r) \dlra{} 1,\]
%\[  0 \dlra{} H \otimes S\pi_{2p}(S^p) \dlra{i} \mathcal{E}(M_r) \dlra{H_p}{\rm Aut}(\phi_r) \dlra{} 1,\]
\[
\begin{diagram}
\divide\dgARROWLENGTH by 2
\node{0} \arrow{e} \node{H_p(M_r) } \arrow{e,t}{} \arrow{s,r}{} \node{{\rm Aut}_{}(M_r)} \arrow{e,t}{\alpha} \arrow{s,r}{\mathcal{E}} \node{{\rm Aut}(\phi_r)} \arrow{s,r}{=} \arrow{e} \node{1}\\
\node{0} \arrow{e} \node{H_p(M_r) \otimes \pi_{2p}(S^p)} \arrow{e,t}{} \node{\mathcal{E}_{}(M_r)} \arrow{e,t}{\alpha} \node{{\rm Aut}(\phi_r)} \arrow{e} \node{1,}
\end{diagram}
\]
to the extensions of $(\ref{gammaexteqn})$ where $c = 12$ or $120$ as $p = 3$ or $7$,
%\[ 0 \dlra{} 2H^* \dlra{i} \Gamma(\psi_F) \dlra{\pi} {\rm Aut}(\phi_r) \dlra{} 1,\]
%\[ 0 \dlra{} 2H^*_{c} \dlra{i} \Gamma(\psi_F, \mathbb{Z}_c) \dlra{\pi} {\rm Aut}(\phi_r) \dlra{} 1,\]
\[
\begin{diagram}
\divide\dgARROWLENGTH by 2
\node{0} \arrow{e} \node{2H^*} \arrow{e,t}{} \arrow{s,r}{} \node{\Gamma(\psi_F, \infty)} \arrow{e,t}{\pi} \arrow{s,r}{} \node{{\rm Aut}(\phi_r)} \arrow{s,r}{=} \arrow{e} \node{1}\\
\node{0} \arrow{e} \node{H^*_c} \arrow{e,t}{} \node{\Gamma(\psi_F, c)} \arrow{e,t}{\pi} \node{{\rm Aut}(\phi_r)} \arrow{e} \node{1.}
\end{diagram}
\]
Hence the extension for ${\rm Aut}(M_r)$ corresponds to $\delta_{\infty}[s(\psi_F)] \in H^2({\rm Aut}(\phi_r); 2H^*)$ and the extension for $\mathcal{E}(M_r)$ corresponds to $\delta_c[s(\psi_F)] \in H^2({\rm Aut}(\phi_r); H_c^*)$.  In particular the extensions above split if and only if and only if $r = 1$.
\end{Corollary}

\begin{proof}
By Lemma \ref{s_Flem2} and Proposition \ref{prop:Brown+} $\text{(iv)}$, $s_F \times \alpha : {\rm Aut}(M_r) \cong \Gamma(\psi_F, \infty)$ defines an isomorphism commuting with $\alpha$ and $\pi$ and this proves the corollary for ${\rm Aut}(M_r)$ since by definition $\Gamma(\psi_F, \infty) \rightarrow \rm{Aut}(\phi_r)$ is the extension defined by $\partial_{\infty}[s(\psi_F)]$.  The corollary for $\mathcal{E}(M_r)$ follows from the fact that $\mathcal{E} : {\rm Aut}(M_r) \rightarrow \mathcal{E}(M_r)$ is onto with kernel $c \cdot H_p(M_r) \subset H_p(M_r)$.  This identifies $\mathcal{E}(M_r)$ with $\Gamma(\psi_F, c)$.  By Proposition \ref{keyextprop} these extensions split if and only if $r  = 1$.
\end{proof}
%%%%%%%%%%%%%%%%%%%%%%%%%%%%%%%%%%%%%%%%

\noindent
\textsc{Diarmuid Crowley\\
School of Mathematical Sciences \\
University of Adelaide \\
Australia, 5005.}\\ \\
{\it E-mail address:} \texttt{diarmuidc23@gmail.com}


\begin{thebibliography}{999}
%%%%%%%%%%%%%%%%%%%%%%%%%%%%%%%%%%%%%%%%
\bibitem[B]{B} H.~J.~Baues, {\em On the group of homotopy equivalences
of a manifold}, Trans. A.M.S. {\bf 348} (1996) no. 12, 4737-4773.

\bibitem[Br]{Br} K.~S.~Brown {\em Cohomology of groups}, Springer-Verlag, New York 1982.

\bibitem[C]{C} J.~Cerf, {\em The pseudo-isotopy theorem for simply connected differentiable manifolds}, Manifolds-Amsterdam (1970), 76-82, LNM {\bf 197}, Springer, Berlin, 1970.

%Manifolds-Amsterdam 1970 (Proc. Nuffic Summer School), 76-82, Lecture Notes in Math., 197, Springer, Berlin, (1970).

%\bibitem[DS]{DS} V.~Markovic and D.~Saric  {em The mapping class group cannot be realized by homeomorphisms}, preprint {\rm http://de.arxiv.org/abs/0807.0182} (2008).

\bibitem[F]{F} D.~Fried, {\em Word maps, isotopy and entropy}, Trans. A.M.S. {\bf 296} (1986) 851-859.

\bibitem[Kre]{Kre} M.~Kreck, {\em Isotopy classes of diffeomorphisms of $(k-1)$-connected almost-parallelizable $2k$-manifolds}, Algebraic topology, Aarhus (1978), 643--663,
LNM {\bf 763}, Springer, Berlin, 1979. 


%Algebraic topology, Aarhus 1978 (Proc. Sympos., Univ. Aarhus, Aarhus, 1978), pp. 643-663, Lecture Notes in Math., 763, Springer, Berlin, (1979).

\bibitem[Kry1]{Kry1} N.~A.~Krylov, {\em Mapping class groups of $(k-1)$-connected almost
parallelizable $2k$-manifolds}, PhD thesis, University of Illinois (2002).

\bibitem[Kry2]{Kry2} N.~A.~Krylov {\em On the Jacobi group and the mapping class group of $S^3 \times S^3$}, Trans. A.M.S. {\bf 355} (2003), 99-117. 

\bibitem[Le]{Le} J.~P.~Levine, {\em Lectures on groups of homotopy spheres}, 
Algebraic and geometric topology, New Brunswick, N.J., (1983), 62-95, 
LNM {\bf 1126}, Springer, Berlin, 1985. 

\bibitem[L\"{u}]{Lu} W.~L\"{u}ck, {\em A basic introduction to surgery theory}, Topology of high-dimensional manifolds, No. 1, 2, Trieste (2001),  1--224, ICTP Lect. Notes, {\bf 9}, Trieste, 2002. 

%ICTP Lecture Notes Series 9, Band 1, of the school ``High-dimensional manifold theory" in Trieste, May/June 2001, Abdus Salam International Centre for Theoretical Physics, Trieste (2002).

\bibitem[Wa]{Wa} C.~T.~C.~Wall, {\em Classification of $(n-1)$-connected
$2n$-manifolds}, Ann. of Math., {\bf 75} (1962) 163-189.

%\bibitem[Wa2]{Wa2} C. T. C. Wall, {\em Classification problems in differential topology - II}, Topology, {\bf 2} (1963) 263-272.

\end{thebibliography}
\end{document}